\documentclass[a4paper,12pt]{article}
\usepackage{amsmath,amssymb,amsfonts,amsthm}
\usepackage[english]{babel}
\usepackage[cm]{fullpage}
\newtheorem{theorem}{Theorem}

    \newtheorem{corollary}[theorem]{Corollary}

    \newtheorem{lemma}[theorem]{Lemma}
    
\newcommand{\supp}{\operatorname{Supp}}

\newcommand{\rank}{\operatorname{rank}}
\newcommand{\ind}{\operatorname{ind}}
\newcommand{\li}{\operatorname{li}}
\newcommand{\rad}{\operatorname{rad}}
\newcommand{\Q}{\mathbb{Q}}
\newcommand{\N}{\mathbb{N}}
\newcommand{\R}{\mathbb{R}}
\newcommand{\Z}{\mathbb{Z}}

\title{Average Multiplicative Order of Finitely Generated Subgroup of Rational Numbers Over Primes}
\author{Cihan Pehlivan\footnote{pehlivan@mat.uniroma3.it}\\
\\
Dipartimento di Matematica, Universit\`a Roma Tre,\\
         Largo S. L. Murialdo, 1, I--00146 Roma Italia}

\begin{document}

\maketitle


\begin{abstract}
Given a finitely generated multiplicative subgroup $\Gamma \subseteq \Q^*$,
assuming the Generalized Riemann Hypothesis,
we determine an asymptotic formula for average over prime numbers, powers of the order of the reduction group modulo $p$.
The problem was considered in the case of rank $1$ by Pomerance and Kurlberg.
In the case when $\Gamma$ contains only positive numbers, we give an explicit expression
for the involved density in terms of an Euler product.
We conclude with some numerical computations.
\end{abstract}

\section{Introduction}
 Let  $\Gamma\subseteq\Q^*$ be a finitely generated multiplicative
subgroup. The
\emph{support} of $\Gamma$ is the (finite) set of primes $p$ for which
the $p$--adic valuation $v_p(g)\neq0$ for some $g\in\Gamma$. We
denote this set by $\supp \Gamma$ and define $\sigma_\Gamma=\prod_{p \in
\supp\Gamma}p.$ For each prime $p\nmid\sigma_\Gamma$, it
is well defined the reduction of $\Gamma$ modulo $p$. That is
\begin{equation}\Gamma_p =
\{g \pmod{p}: g \in \Gamma\}.\end{equation}
For simplicity, when $p$
does divide the support of $\Gamma$ , we let $\Gamma_p = \{1\}$.
We also denote by $\Q(\zeta_k, \Gamma^{1/k})$ the extension of the cyclotomic field $\Q(\zeta_k)$ obtained by adding the $k$--th roots of all the
elements in $\Gamma$. We denote Jordan's totient function by
\begin{equation}J_r(m):=m^r\prod_{\ell\mid m}\left(1-\frac1{\ell^r}\right)\end{equation} and sum of $t$--th power of positive divisors of $n$ by
\begin{equation}\sigma_t(n):=\sum_{d|n} d^t.\end{equation}

\begin{theorem}\label{Main Theorem} Let $\Gamma\subseteq\Q^*$ be a finitely generated multiplicative
subgroup with rank $r\geq 2$ and assume that the Generalized Riemann Hypothesis
holds for $\Q(\zeta_k, \Gamma^{1/k})$ ($k\in\N$). Let
\begin{equation}C_{\Gamma,t}:=\sum_{k\geq1}\frac{J_t(k) (\operatorname{rad}(k))^t (-1)^{\omega(k)} }{k^{2t} [\Q(\zeta_k, \Gamma^{1/k}): \Q]}\end{equation} where $\rad(k)$ denotes the product of distinct prime numbers dividing $k$.
Then the series $C_{\Gamma,t}$ converges absolutely and as $x\rightarrow\infty$,
\begin{equation}\sum_{p\leq x}|\Gamma_p|^t= \li(x^{t+1})\left( C_{\Gamma,t} +O_\Gamma\!\left(\frac{\log\log x}{(\log x)^{r}}\right)\right)\end{equation}
 and the constant implied by the $O_\Gamma$--symbol may depend on $\Gamma$.
\end{theorem}
Further on, for the case $t=1$ we use $C_\Gamma$ instead of $C_{\Gamma,1}$. Kurlberg and Pomerance in \cite{arnold} consider the case when $\Gamma=\langle g\rangle$ has rank $1$.
In the special case when $\Gamma\subset\Q^+$, we express
the value of $C_\Gamma$ as an Euler product. To this purpose, we introduce some
notations:

\begin{itemize}
 \item If $\eta\in\Q^*$, by  $\delta(\eta)$ we denote the \emph{field discriminant} of $\Q(\sqrt{\eta})$.
\item For any $k\in\N^+$,
$\Gamma(k)=\Gamma \cdot {\Q^{*}}^k/{\Q^{*}}^k.$
\end{itemize}
For any square-free integer $\eta$,
let
$$\label{teta}
t_\eta=\begin{cases}
\infty &\text{if for all }t\geq0,\ \eta^{2^{t}}{\Q^*}^{2^{t+1}}\not\in\Gamma(2^{t+1})\\
\min\{t\in \N : \eta^{2^t}{\Q^*}^{2^{t+1}}\in\Gamma(2^{t+1})\} &\text{otherwise.}\end{cases}
$$

We will show the following:
\begin{theorem}\label{density}
Assume that $\Gamma $ is a finitely generated subgroup of $\Q^+$. Then
\begin{eqnarray}C_{\Gamma,t}&&=\prod_{p}\left(1-\sum_{\substack{\alpha\geq1}}\frac{p^t-1}{p^{\alpha(t+1)-1} |\Gamma(p^{\alpha})|(p-1)}\right)\nonumber\\
&&\times\left(1+
\sum_{\substack{\eta\mid \sigma_\Gamma\\ \eta\neq1}}
S_\eta
\prod_{p\mid 2\eta}\left(1-\left(\sum_{\substack{\alpha\geq1}}
\frac{p^t-1}{p^{\alpha(t+1)-1}|\Gamma(p^{\alpha})|(p-1)}\right)^{-1}\right)^{-1}
\right)\end{eqnarray}
where
\begin{equation}S_\eta=\frac{\displaystyle\sum_{\alpha\ge\gamma_\eta}
\frac{2^t-1}{2^{\alpha(t+1)-1}|\Gamma(2^\alpha)|}}{
\displaystyle\sum_{\alpha\ge1}
\frac{2^t-1}{2^{\alpha(t+1)-1}|\Gamma(2^\alpha)|}}\end{equation}
 and
$\gamma_\eta=\max\{1+t_\eta,v_2(\delta(\eta))\}.$
\end{theorem}
A calculation shows that, in the case when $\Gamma=\langle g\rangle$,
the above expression for $C_{\langle g\rangle}$ coincides with that of Kurlberg and Pomerance.
In the special case when $\Gamma$ consists of prime numbers and $t=1$, the above formula can be considerably simplified:

\begin{corollary}\label{densityPrime}
  Let $\Gamma=\langle p_1,\ldots,p_r\rangle$ where all the $p_i$'s are prime numbers and $r\ge1$, with the notation above,
 we have
\begin{eqnarray}C_{\langle p_1,....,p_r\rangle}&&=\prod_{p}\left(1-\frac p{p^{r+2}-1}\right)\nonumber\\
&&\times \left(1+\sum_{\substack{\eta\mid p_1\cdots p_r\\ \eta\neq1}}\frac{1}{2^{\max\{0,v_2(\delta(\eta)/2)\}(r+2)}}\prod_{\ell\mid 2\eta}
\frac{\ell}{\ell+1-\ell^{r+2}}\right).\end{eqnarray}
\end{corollary}
The quantity
\begin{equation}C_r= \prod_{p}\left(1-\frac{p}{p^{r+2}-1}\right)\end{equation}
can be computed with arbitrary precision:
\begin{center}
 \begin{tabular}{ll}
 $r$ & $C_r$\\
  \hline
1 & $0.57595996889294543964316337549249669251\cdots$ \\
2 & $0.82357659279814332395380438513901050177\cdots$ \\
3 & $0.92190332088740008067545348360869076931\cdots$ \\
4 & $0.96388805107176946676374437726734997946\cdots$ \\
5 & $0.98282912014687261524345691713313004185\cdots$ \\
6 & $0.99168916383630008819101294319807859837\cdots$ \\
7 & $0.99593155027181927318700546733612700362\cdots$ \\
8 & $0.99799372275691129752727433560285572887\cdots$ \\
9 & $0.99900593591154969071253065973483263501\cdots$ \\
10 & $0.99950593624928276115384423618416539651\cdots$
 \end{tabular}
\end{center}

Furthermore, we have the following corollary.
\begin{corollary} \label{naive_x_rational}
Let $\Gamma $ be a finitely generated subgroup of $\Q^+$  with rank $r$. Then $C_\Gamma$ is a non zero rational multiple of $C_r$.
\end{corollary}
We conclude the paper with some numerical evidence. Complete account of the results in this stream we reefer the survey of P. Moree \cite{Msurvey}.
\section{Notational conventions}
Throughout the paper, the letter $p$ always denote \emph{prime numbers}.
As usual, we use $\pi(x)$ to denote
the \emph{number of $p \le x$} and
\begin{equation}\li(x)=\int_2^x \frac{dt}{\log t}\end{equation}
denotes the \emph{logarithmic integral} function. The invariant $\Delta_r(\Gamma)$ of a multiplicative subgroup
$\Gamma\subseteq\Q^*$ with $\rank_\Z(\Gamma)=r$ is defined as the
greatest common divisor of all the minors of size $r$ of the
relation matrix of the group of $\Gamma$ (see
\cite[Section 3.1]{CP} for some details).

$\varphi$ and $\mu$ are respectively the \emph{Euler} and the
\emph{M\"obius} functions. An integer is said \emph{squarefree}
if it is not divisible for the square of any prime number.
If $\eta\in\Q^*$, by  $\delta(\eta)$ we denote the \emph{field discriminant} of $\Q(\sqrt{\eta})$.
 So, if $\eta\in\Z$ is square-free,
 $\delta(\eta)=\eta$ if $\eta\equiv1\ (\bmod\ 4)$, and $\delta(\eta)=4\eta $ otherwise. For $\alpha\in\Q^*$ we denote by $v_\ell(\alpha)$ the \emph{$\ell$--adic valuation} of $\alpha$.

For functions $F$ and $G>0$ the notations $F=O(G)$ and $F\ll G$
are equivalent to the assertion that the inequality $|F|\le c\,G$ holds
with some constant $c>0$. We write $F\sim G$ if $\lim_{x\rightarrow \infty}\frac{F(x)}{G(x)}=1.$  In what follows, all constants implied by the
symbols $O$ and $\ll$ may depend (when obvious) on the small real
parameter $\epsilon$ but are absolute otherwise; we write $O_\lambda$ and $\ll_\lambda$
 to indicate that the implied constant depends on a given
parameter $\lambda$. We also define the index of subgroup $\ind(\Gamma_p)= \frac{p-1}{|\Gamma_p|}.$

\section{Lemmata}

In this section we present some results which we need for proof of the main theorem.
The following Lemma describes explicitly the degree of $[\Q(\zeta_k, \Gamma^{1/k}):\Q]$ (see \cite[Lemma 1 and Corollary 1]{pa2}).
\begin{lemma}\label{degree} Let $k\geq1$ be an integer. With the notation above, we have
$$\left[\Q(\zeta_k, \Gamma^{1/k}):\Q(\zeta_k)\right]=
|\Gamma(k)|/|\widetilde\Gamma(k)|$$
 where
\begin{equation}\widetilde{\Gamma}(k)=(\Gamma\cap\Q(\zeta_k)^{2^{v_{2}(k)}})\cdot{\Q^*}^{2^{v_{2}(k)}}/{\Q^*}^{2^{v_{2}(k)}}.\end{equation}
Furthermore, in the special case when $\Gamma\subset\Q^+$,
\begin{equation}\widetilde{\Gamma}(k)=\{\eta\mid \sigma_\Gamma, \eta^{2^{v_2(k)-1}}{\Q^*}^{2^{v_2(k)}}\in\Gamma(2^{v_2(k)}),\
\delta(\eta)\mid k\}.\end{equation}
\end{lemma}

The following statement
is obtained using the effective version of the Chebotarev Density Theorem due to Serre (see \cite[Theorem~4]{serre}).

\begin{lemma} [Chebotarev Density Theorem]\label{cdt} Let
$\Gamma \subset \Q^*$ be a finitely generated subgroup of rank $r$ and $k\in\N^+$.
The GRH  for the Dedekind zeta function
of $\Q(\zeta_k, \Gamma^{1/k})$ implies that \begin{equation}\#\left\{p\leq x: p\not\in\supp\Gamma,\
k\mid\ind(\Gamma_p)\right \}=
\frac{\li(x)}{\left[\Q(\zeta_k,
\Gamma^{1/k}):\Q\right]}+O\!\left(\sqrt{x}\log(xk^{r+1}\sigma_\Gamma)\right).
\end{equation}
\end{lemma}

The explicit formula for the degree $\left[\Q(\zeta_k,
\Gamma^{1/k}):\Q\right]$ can be found in \cite[Lemma~1]{pa2}.
The next results follows from Lemma~\ref{degree} (see \cite[Equation~7]{pa2}).

\begin{corollary}\label{corr-1}
Let $\Gamma\subset\Q^*$ be a subgroup of $r=\rank_\Z(\Gamma)$ and $k\in\N$. Then
\begin{equation}
2k^r\ge [\Q(\zeta_k, \Gamma^{1/k}):\Q(\zeta_k)] \ge
\frac{(k/2)^r}{\Delta_r(\Gamma)}.
\end{equation}
\end{corollary}

\noindent Next Lemma is implicit in the work of C.~R.~Matthews (see \cite{MC}).
\begin{lemma}\label{Matthews Lemma} Assume that $\Gamma\subseteq\Q^*$ is a multiplicative subgroup of rank $r\geq 2$
and assume that $(a_1,\ldots, a_r)$ is a $\Z$--basis of $\Gamma$.
Let $t\in\R$, $t>1$. We have the following estimate
\begin{equation}
\#\left\{p\not\in\supp\Gamma: |\Gamma_p|\le t\right\}\ll_\Gamma\frac{t^{1+
1/r}}{\log t}.
\end{equation}
\end{lemma}

\begin{theorem}\label{Main Lemma} Assume the GRH. Let $\Gamma$ be a multiplicative
subgroup of $\Q^*$ of rank $r \geq 2$.
Then, for  $1\le L\le\log x$,
 we have
\begin{equation}\#\left\{ p\leq x: p\not\in\supp\Gamma, |\Gamma_p|\leq \frac{p-1}{L}\right\}\ll_\Gamma \frac{\pi(x)}{L^r}.\end{equation}
\end{theorem}

The proof of the above is routine and easier than the main theorem in \cite{H} and to \cite[Theorem 6]{arnold}. Hence we will skip
some of the details.
\begin{proof}

Let $t$, $L\leq t\leq x$ be a parameter that will be chosen later.
\begin{itemize}
\item \emph{first step}:
First consider primes $p\not\in\supp\Gamma$ such that $|\Gamma_p|\leq \frac{p-1}{t}$. By Lemma \ref{Matthews Lemma}, we have
\begin{equation}\#\left\{p\not\in\supp\Gamma: |\Gamma_p|\le \frac{x}{t}\right\}\ll_\Gamma\frac{(x/t)^{1+
1/r}}{\log (x/t)}.\end{equation}

\item \emph{second step}:
Next consider the primes $p\not\in\supp\Gamma$ such that there exists a prime $q$, $L\le q\le t$ such that
$q\mid\ind(\Gamma_p)=\frac{p-1}{|\Gamma_p|}$.
If we apply Lemma~\ref{cdt}, we obtain
\begin{eqnarray}
\#\{p\le x: p\not\in\supp\Gamma, q \mid \ind(\Gamma_p)\}&=&
\frac{\li(x)}{\left[\Q(\zeta_q,
\Gamma^{1/q}):\Q\right]}+O_\Gamma\left(\sqrt{x}\log(xq)\right)\nonumber\\
&\ll_\Gamma& \frac{\pi(x)}{q^r \varphi(q)}+ \sqrt{x}\log(x q)
\end{eqnarray}
where in the latter estimate we have applied Corollary~\ref{corr-1}.
If we sum the above over primes
 $q$: $L\le q\le t$, we obtain
 \begin{eqnarray}
 &&\#\{p\le x: p\not\in\supp\Gamma, \exists q \mid \ind(\Gamma_p), L\le q\le t\}\nonumber\\
 &&\ll_\Gamma\sum_{\substack{q\text{ prime}\nonumber\\  L\le q\le t}}
\left(\frac{\pi(x)}{q^r \varphi(q)}+ \sqrt{x}\log(x q)\right)\ll_\Gamma\frac{\pi(x)}{L^r}+ x^{1/2}t\log x.
 \end{eqnarray}

\item \emph{third step}:
The primes $p$ that were not counted in previous steps, have the property
that all the prime divisors of $\ind(\Gamma_p)$ belong to the interval $[ 1, L]$. Hence, for such primes $p$,
$\ind(\Gamma_p)$ is divisible for some integer $d$ in $[ L, L^2]$.

Applying again Lemma~\ref{cdt} and Corollary~\ref{corr-1}, and taking the  sum over $d$  we deduce that the total
number of such primes is
\begin{equation}\ll_\Gamma \sum_{\substack{ d\in\N\\ L<d\le L^2}}\left(\frac{\pi(x)}{d^r \varphi(d)}+ x^\frac1{2}\log(xd)\right)\ll_\Gamma
\frac{\pi(x)}{L^r}+ x^{1/2} L^2\log x.\end{equation}
\end{itemize}
A choice of $t=\frac {x^{1/2}}{L^r \log^2 x}$ allows us to conclude the proof.
\end{proof}

The Theorem of Wirsing \cite{Wirsing} is formulated as follows.
\begin{lemma}\label{multiplicative}
Assume that a real valued multiplicative function $h(n)$ satisfies the following conditions.
\begin{itemize}
\item $h(n)\geq0, n=1,2,..;$
\item $h(p^n)\leq c_1 {c_2}^v, v=2,3...$, for some constants $c_1, c_2$ with $c_2<2;$
\item there exists a constant $\tau>0$ such that
\begin{equation} \sum_{p\leq x}h(p)= (\tau + o(1)) \frac{x}{\log x}.\end{equation}
\end{itemize}
Then for any $x\geq0$,
\begin{equation}\sum_{n\leq x}h(n)= \left( \frac{1}{e^{\gamma\tau}\Gamma(\tau)}+o(1)\right) \frac{x}{\log x} \prod_{p\leq x}\sum_{\nu\geq0}\frac{h(p^\nu)}{p^\nu}\end{equation}
where $\gamma$ is the Euler constant, and
\begin{equation} \Gamma(s)= \int_0^\infty e^{-t}t^{s-1} dt\end{equation} is the gamma function.
\end{lemma}

\section{Proof of the Theorem $2$}
\begin{proof}[Proof of Theorem ~\ref{density}]We start by splitting the sum $C_{\Gamma,t}$ as
\begin{equation}C_{\Gamma,t}:=\sum_{k\geq1}\frac{J_t(k) (\operatorname{rad}(k))^t (-1)^{\omega(k)} }{k^{2t} [\Q(\zeta_k, \Gamma^{1/k}): \Q]}=A_1+A_2\end{equation}
where $A_1$ is the sum of the terms corresponding to odd values of $k$ and $A_2$ is the sum of the terms corresponding to even values of $k$.
Note that if $\Gamma \subseteq \Q^+$ by Lemma \ref{degree} we have
\begin{equation}[\Q(\zeta_k, \Gamma^{1/k}): \Q]=\frac{\varphi(k)|\Gamma(k)|}{|\widetilde{\Gamma}(k)|}\end{equation}
where, if $k$ is even,
\begin{equation}\widetilde{\Gamma}(k)=\{\eta\mid \sigma_\Gamma, \eta^{2^{v_2(k)-1}}{\Q^*}^{2^{v_2(k)}}\in\Gamma(2^{v_2(k)}),\
\delta(\eta)\mid k\} \end{equation}
while if $k$ is odd $\widetilde{\Gamma}(k)=\{1\}$. We define $$f_t(k)=\frac{J_t(k) (\operatorname{rad}(k))^t (-1)^{\omega(k)} }{k^{2t} \varphi(k)|\Gamma(k)| }.$$
Note that if $D\in\N^+$ is even, since $f_t(k)$ is multiplicative in $k$, then
\begin{equation}\sum_{\substack{k\geq1\\ \gcd(k,D)=1}}f_t(k)=
\prod_{p\nmid D}\left(1+\sum_{\alpha\ge1}f_t(p^\alpha)\right)=\prod_{p\nmid D}\left(1-\sum_{\alpha\ge1}\frac{p^t-1}{p^{\alpha(t+1)-1}|\Gamma(p^\alpha)|(p-1)}\right).\end{equation}
Therefore, we have the identity
\begin{equation}A_1= \prod_{\substack{p>2}}\left(1+\sum_{\alpha\ge1}f_t(p^\alpha)\right)= \prod_{\substack{p>2}}\left(1-\sum_{\alpha\ge1}\frac{p^t-1}{p^{\alpha(t+1)-1}|\Gamma(p^\alpha)|(p-1)}\right).\end{equation}
We can write $A_2$ as,
\begin{eqnarray}
A_2&=&\sum_{\eta\mid \sigma_\Gamma}\sum_{\substack{k\geq1,  2\mid k\\ \widetilde{\Gamma}(k)\ni\eta}}\frac{J_t(k) (\operatorname{rad}(k))^t (-1)^{\omega(k)} }{k^{2t} \varphi(k)|\Gamma(k)|}\nonumber\\
&=&\sum_{\eta\mid \sigma_\Gamma}
\sum_{\substack{\alpha\ge1\\ \eta^{2^{\alpha-1}}{\Q^*}^{2^{\alpha}}\in\Gamma(2^{\alpha})}}\sum_{\substack{k\geq1\\  v_2(k)=\alpha\\ \delta(\eta)\mid k}} f_t(k)\nonumber\\
&=&\sum_{\eta\mid \sigma_\Gamma}
\sum_{\substack{\alpha\ge1\\ \eta^{2^{\alpha-1}}{\Q^*}^{2^{\alpha}}\in\Gamma(2^{\alpha})
\\ \alpha\ge v_2(\delta(\eta))}}
\frac{-(2^t-1))}{2^{\alpha(t+1)-1}|\Gamma(2^\alpha)|}
\sum_{\substack{k\geq1\\  2\nmid k\\ \delta(\eta)\mid 8k}}
f_t(k).
\end{eqnarray}
Now write $\delta(\eta)=2^{v_2(\delta(\eta))}M$. Then
\begin{align}
&\sum_{\substack{k\geq1\\  2\nmid k\\ \delta(\eta)\mid 8k}}
f_t(k)=
\prod_{\substack{p>2\\ p\nmid M}}\left(1+\sum_{\substack{\alpha\geq1}}f_t(p^\alpha)\right)\prod_{\substack{p>2\\ p\mid M}}\left(\sum_{\substack{\alpha\geq1}}f_t(p^\alpha)\right)\nonumber\\
&=A_1\prod_{\substack{p>2\\ p\mid M}}\left(1+\sum_{\substack{\alpha\geq1}}
f_t(p^\alpha)\right)^{-1}\left(\sum_{\substack{\alpha\geq1}}f_t(p^\alpha)\right)
\end{align}
Hence, if $t_\eta$ is the quantity defined in (\ref{teta}), then
$$C_{\Gamma,t}:=A_1\times\left(1+
\sum_{\eta\mid \sigma_\Gamma}
\sum_{\substack{\alpha\ge1\\ \alpha\ge t_\eta+1
\\ \alpha\ge v_2(\delta(\eta))}}
\frac{-(2^t-1)}{2^{\alpha(t+1)-1}|\Gamma(2^\alpha)|}
\prod_{\substack{p>2\\ p\mid M}}\left(1+\left(\sum_{\substack{\alpha\geq1}}
f_t(p^\alpha)\right)^{-1}\right)^{-1}
\right).$$
Now let
$$\delta_\Gamma:=\prod_{p\text{ prime}}\left(1+\sum_{\substack{\alpha\geq1}}f_t(p^\alpha)\right)=\prod_{p\text{ prime}}\left(1-\sum_{\substack{\alpha\geq1}}\frac{p^t-1}{p^{\alpha(t+1)-1}|\Gamma(p^\alpha)|(p-1)}\right)$$
and deduce that
$$C_{\Gamma,t}=\delta_\Gamma\left(1+
\sum_{\substack{\eta\mid \sigma_\Gamma\\ \eta\neq1}}
\frac{\displaystyle\sum_{\alpha\ge\gamma_\eta}
\frac{2^t-1}{2^{\alpha(t+1)-1}|\Gamma(2^\alpha)|}}{
\displaystyle\sum_{\alpha\ge1}
\frac{2^t-1}{2^{\alpha(t+1)-1}|\Gamma(2^\alpha)|}}
\prod_{p\mid 2\eta}\left(1+\left(\sum_{\substack{\alpha\geq1}}
f_t(p^\alpha)\right)^{-1}\right)^{-1}
\right)$$
where $\gamma_\eta=\max\{1+t_\eta,v_2(\delta(\eta))\}$ and this completes the proof.
\end{proof}

\section{Proof of Corollary $3$}
Let $\Gamma$ be a finitely generated subgroup of $\Q^+$ of rank $r$ and let $(a_1,...,a_r)$ be a $\Z-$basis of $\Gamma$.
We write $\supp(\Gamma)=\{p_1,...,p_s\}$.
Then we can construct the $s\times r-$matrix with coefficients in $\Z:$

\begin{equation}M(a_1,...,a_r)=A= \begin{pmatrix}
\alpha_{1,1} & ... & \alpha_{1,r} \\
\vdots & & \vdots & \\
\alpha_{s,1} & ... & \alpha_{s,r}
\end{pmatrix}\end{equation}
defined by the property that $|a_i|= (p_1)^{\alpha_{1,i}}...(p_s)^{\alpha_{s,i}}$. It is clear that $s\ge r$ and that
the rank of the matrix $M(a_1,...,a_r)$ equals $r$. For all $i=1,...,r$ we define the $i-$th  exponent of $\Gamma$ by
$$\Delta_i=\Delta_i(\Gamma)= \gcd( \operatorname{det}A: A\ \text{is a}\ i\times i\ \text{minor of}\ M(a_1,...,a_r))$$
and we also set $\Delta_0=1$.
 For $m\in\N$, we have (see \cite[Proposition 2]{CP} )
$$ | \Gamma(m)|= \frac{
m^r}{\gcd(m^r,m^{r-1}\Delta_1,...,m \Delta_{r-1}, \Delta_r)}$$
and in particular, for every prime power $p^\alpha$, we have
$$|\Gamma(p^\alpha)|= p^{\max\{0,\alpha-v_p(\Delta_1),...,(r-1)\alpha-v_p(\Delta_{r-1}),r\alpha-v_p(\Delta_r)\}}.$$
\begin{proof}[Proof of Corollary ~\ref{densityPrime}] Let $\Gamma$ be generated by prime numbers $p_1,....,p_r$, since $\Delta_i$'s are $1$ we have $|\Gamma(k)|=k^r$ and $t_\eta=0$ for all $\eta\mid\sigma_\Gamma=p_1\cdots p_r$ and
$$\gamma_\eta=\begin{cases}
  1& \text{if }\eta\equiv1\bmod4\\
  2 &\text{if }\eta\equiv3\bmod4\\
  3 &\text{if }\eta\equiv2\bmod4.
              \end{cases}
$$
Furthermore
$$\sum_{\alpha\ge\gamma_\eta}
\frac{1}{2^{2\alpha-1}|\Gamma(2^\alpha)|}=\frac{1}{2^{(\gamma_\eta-1)(r+2)}}\sum_{\alpha\ge1}
\frac{1}{2^{2\alpha-1}|\Gamma(2^\alpha)|}$$
and since $|\Gamma(k)|= k^r$ for all $k\in\N^+$, we have that
$$\sum_{\substack{\alpha\geq1}}\frac{1}{p^{2\alpha-1} |\Gamma(p^{\alpha})|}
=\frac{p}{p^{r+2}-1}.$$
Hence, if we let
$$C_r=\prod_{p}\left(1-\frac p{p^{r+2}-1}\right),$$
then
$$C_{\langle p_1,....,p_r\rangle}=C_r\left(1+\sum_{\substack{\eta\mid p_1\cdots p_r\\ \eta\neq1}}\frac{1}{2^{(\gamma_\eta-1)(r+2)}}\prod_{\ell\mid 2\eta}
\frac{\ell}{\ell+1-\ell^{r+2}}\right)$$
and this completes the proof.
\end{proof}

\section{Proof of Corollary $4$}

\begin{proof}[Proof of Corollary ~\ref{naive_x_rational}]  If we set $k_p=\max\{v_p(\Delta_r/\Delta_{r-1}),\cdots,v_p(\Delta_1/\Delta_0)\}$
then for $\alpha\ge k_p$, $|\Gamma(p^{\alpha})|=p^{r\alpha-v_p(\Delta_r)}$.
 Hence
 $$\sum_{\alpha\geq1}\frac{1}{p^{2\alpha-1}|\Gamma(p^{\alpha})|}=
 \sum_{\alpha=1}^{k_p}\frac{1}{p^{2\alpha-1}|\Gamma(p^{\alpha})|}+
 \frac{p^{v_p(\Delta_r)+1-(r+2)k_p}}{p^{r+2}-1}
 \in\Q.$$
 In particular, if $p\nmid \Delta_r$, then $k_p=0$ and
 $|\Gamma(p^{\alpha})|=p^{\alpha r}$ for all $\alpha\ge0$ and
$$\sum_{\alpha\geq1}\frac{1}{p^{2\alpha-1}|\Gamma(p^{\alpha})|}=
\frac{p}{p^{r+2}-1}.$$
Therefore
$$C_\Gamma=
r_\Gamma\prod_{p\nmid\Delta_r}\left(1-\frac{p}{p^{r+2}-1}\right)$$
where
\begin{eqnarray}r_\Gamma&&=\prod_{p\mid\Delta_r}\left(1-\sum_{\substack{\alpha\geq1}}\frac{1}{p^{2\alpha-1} |\Gamma(p^{\alpha})|}\right)\nonumber\\
&&\times\left(1+
\sum_{\substack{\eta\mid \sigma_\Gamma\\ \eta\neq1}}
S_\eta
\prod_{p\mid 2\eta}\left(1-\left(\sum_{\substack{\alpha\geq1}}
\frac{1}{p^{2\alpha-1} |\Gamma(p^{\alpha})|}\right)^{-1}\right)^{-1}
\right)\in\Q.\end{eqnarray}
Finally $C_\Gamma$ is a rational multiple of
$$C_r= \prod_{p}\left(1-\frac{p}{p^{r+2}-1}\right)$$
and this concludes the proof.
\end{proof}

\section{Proof of Theorem $1$}
The proof use the methods of Kurlberg and Pomerance \cite[Theorem
2]{arnold}.
\begin{proof}[Proof of Theorem ~\ref{Main Theorem}] Let $z=\log x$. We have
$$\sum_{p\leq x}|\Gamma_p|^t=\sum_{\substack{p\leq x\\ \ind(\Gamma_p) \leq z}}|\Gamma_p|^t+\sum_{\substack{p\leq x\\\ind(\Gamma_p) > z}}|\Gamma_p|^t= A+E,$$
say. We write $|\Gamma_p|^t=\frac{(p-1)^t}{\ind^t(\Gamma_p)} $ and use the identity $\frac1{\ind^t(\Gamma_p) }=\sum_{uv|\ind(\Gamma_p) }\frac{\mu(v)}{u^t}$, after splitting the sum we have
\begin{eqnarray*}
 A&=&\sum_{p\leq x}(p-1)^t \sum_{\substack{uv|\ind(\Gamma_p)\\ uv\leq z}}\frac{\mu(v)}{u^t}-\sum_{\substack{p\leq x\\ \ind(\Gamma_p) > z}}(p-1)^t \sum_{\substack{uv|\ind(\Gamma_p)\\ uv\leq z}}\frac{\mu(v)}{u^t}\\&=&A_1-E_1,
 \end{eqnarray*}
say. The main term is $A_1$, after switching the summation
 and applying partial summation and using Lemma~\ref{cdt} on GRH, we have
$$A_1=\li(x^{t+1})\sum_{\substack{uv\leq z}}\frac{\mu(v)}{u^t [\Q(\zeta_{uv}, \Gamma^{1/uv}):\Q]}+
O\!\left( x^{t+\frac{1}{2}}\log x \sum_{n\leq z}\left|\sum_{uv=n}\frac{\mu(v)}{u^t}\right|\right).$$
The inner sum in the $O$-term is bounded by $\frac{\varphi(n)}{n}$ so that the $O$-term above
is $O\!\left( x^{t+\frac{1}{2}}\log^2(x)\right)$.
Next we use the elementary fact
$ J_t(\operatorname{rad}(k))= J_t(k)\left(\frac {\rad(k)}{k}\right)^t$
and $\sum_{v|k}\mu(v)v^t=\prod_{p|k}{(1-p^t)}=(-1)^{\omega(k)} J_t(\rad(k))= (-1)^{\omega(k)} \frac {J_t(k)(\rad(k))^t}{k^t}.$
So
$$\sum_{uv=k}\frac{\mu(v)}{u^t [\Q(\zeta_{uv}, \Gamma^{1/uv}):\Q]}=\sum_{v|k}\frac{\mu(v)v^t}{k^t[\Q(\zeta_{k}, \Gamma^{1/k}):\Q]}=\frac{(-1)^{\omega(k)} J_t(k)(\rad(k))^t}{k^{2t}[\Q(\zeta_{k}, \Gamma^{1/k}):\Q])}.$$
Let $C_{\Gamma,t}:=\sum_{k\geq1}\frac{J_t(k) (\rad(k))^t (-1)^{\omega(k)} }{k^{2t}[\Q(\zeta_{k}, \Gamma^{1/k}):\Q]}$,
 after applying Corollary~\ref{corr-1}, finally we have
$$A_1= \li(x^{t+1})\left( C_{\Gamma,t}+ O\!\left(\frac{1}{z^{r}} \right)\right).$$
It remains to estimate the error terms $E$ and $E_1$. Applying Theorem~\ref{Main Lemma}:
$$E\ll \frac{x^t}{z^t} \frac{\pi(x)}{z^{r}}.$$
In order to estimate $E_1$, we calculate
$$\left|\sum_{\substack{uv|n\\ uv\leq z}}\frac{\mu(v)}{u^t}\right| \leq \sum_{\substack{u|n }}\frac{1}{u^t} \sum_{\substack{v|n\\v\leq z}}1\leq \frac{\tau(n)\sigma_t(n)}{n^t},$$ so
  $$E_{1}\leq \sum_{\substack{z<n}}\frac{\tau(n)\sigma_t(n)}{n^t}\sum_{\substack{p\leq x\\ n\mid\text{ind}(\Gamma_p) }}(p-1)^t.$$

\noindent Then applying Lemma~\ref{cdt} and Corollary~\ref{corr-1} we obtain that

\begin{eqnarray*}
  E_{1}&\ll& x^t \pi(x) \sum_{\substack{z<n}}\frac{\tau(n)\sigma_t(n)}{n^t \varphi(n) n^r}.
 \end{eqnarray*}


 \noindent Let $g(n):=\frac{\tau(n)\sigma_t(n)}{n^{t-1}\varphi(n)}$,
 $\sum_{p\leq x}g(p)= (2 + o(1)) \frac{x}{\log x}$. Using Lemma~\ref{multiplicative} (for in our case $\tau$ is $2$), we have
 \begin{eqnarray*}
\sum_{n\leq x}g(n)&=&\left(\frac{1}{e^{\gamma 2}}+o(1)\right) \frac{x}{\log x}  \prod_{p\leq x}
\left(1+ \frac{p}{(p-1)(p^t-1)}\sum_{\nu\geq1}\frac{(\nu+1)(p^{\nu t+t}-1)}{p^{\nu t+\nu} }\right).
\end{eqnarray*}
\noindent To make the product convergent we add a correction factor, and invoke Merten's third formula, we have
\begin{eqnarray*}
\sum_{n\leq x}g(n)\sim & x \log x .\end{eqnarray*}
 Let $G(n):=\sum_{n\leq x}g(n)$ using partial summation, we have
$$\sum_{z<n}\frac{g(n)}{n^{r+1}}=\lim_{T\rightarrow\infty}
\left(\frac{G(T)}{T^{r+1}}-\frac{G(z)}{z^{r+1}}\right)-\int_z^\infty G(u) \frac{d}{du}\left(\frac{1}{u^{r+1}}\right)\ll \frac{\log z}{z^r}.$$
Therefore, we obtain $$E_{1}\ll x^t \pi(x) \frac{\log z}{z^r}.$$
\noindent We have chosen $z=\log x$, finally we have
$$ \sum_{p\leq x}|\Gamma_p|^t= \li(x^{t+1})C_{\Gamma,t}+O\!\left( \frac{x^{t+1}\log\log x}{(\log x)^{r+1}}\right).$$
\end{proof}

\section{Numerical Examples}
In this section we compare some numerical data.
The tables compares the value of $C_\Gamma$ as predicted by Corollary~\ref{densityPrime} with
$$A_\Gamma=\frac{\displaystyle\sum_{p \le10^{10}} |\Gamma_p|}{\displaystyle\sum_{p\le10^{10}}p}.$$
We consider the following cases:
\begin{itemize}
 \item $\Gamma_r=\langle 2,...,p_r\rangle$, the group generated by the first $r$ primes
 \item $\Gamma'_r=\langle3,...,p_{r+1}\rangle$, the group generated by the first $r$ odd primes.
 \item $\Gamma''_r=\langle5,...,p''_r\rangle$, the group generated by the first $r$ primes congruent to $1$
modulo $4$.

\end{itemize}

\begin{center}
\begin{scriptsize}
\begin{tabular}{c|l|l|l|l|l|l|l}
    $r$          &1           &2           &3           &4           &5           &6           &7\\ \hline
  $A_{\Gamma_r}  $&0.5723625220&0.8234145762&0.9219692467&0.9638944667&0.9828346715&0.9916961670&0.9959388895\\
  $C_{\Gamma_r}  $&0.5723602190&0.8234094709&0.9219688310&0.9638925514&0.9828293379&0.9916891587&0.9959315465\\\hline
  $A_{\Gamma'_r} $&0.5797271743&0.8249081874&0.9220326599&0.9639044730&0.9828352799&0.9916947130&0.9959372205\\
  $C_{\Gamma'_r} $&0.5797162295&0.8249060912&0.9220306381&0.9639002343&0.9828302996&0.9916892783&0.9959315614\\\hline
  $A_{\Gamma''_r}$&0.5856374600&0.8246697078&0.9220170449&0.9639045923&0.9828329969&0.9916930151&0.9959357111\\
  $C_{\Gamma''_r}$&0.5856399683&0.8246572843&0.9220082264&0.9638982767&0.9828301305&0.9916892643&0.9959315465\\\hline

\end{tabular}
\end{scriptsize}
\end{center}

\textsc{Acknowledgement:} This paper is part of the Doctorate thesis at the Universit\`{a} Roma Tre.

\bibliographystyle{plain}

\end{document}